\documentclass[12pt,english,a4paper]{amsart}

\usepackage{fullpage}
\usepackage{hyperref}
\usepackage{amssymb}
\usepackage{enumerate}

\numberwithin{equation}{section}

{\theoremstyle{plain}
\newtheorem{theorem}{Theorem}[section]    
\newtheorem{lemma}[theorem]{Lemma}       
\newtheorem{proposition}[theorem]{Proposition}}

{\theoremstyle{remark}  
\newtheorem{remark}[theorem]{Remark}}

\begin{document}

\author{Arno Fehm}
\email{arno.fehm@tu-dresden.de}
\address{Institut f\"ur Algebra, Fakult\"at Mathematik, TU Dresden, 01062 Dresden, Germany}

\author{Fran\c{c}ois Legrand}
\email{francois.legrand@tu-dresden.de}
\address{Institut f\"ur Algebra, Fakult\"at Mathematik, TU Dresden, 01062 Dresden, Germany}

\title{A note on finite embedding problems \\ with nilpotent kernel}

\maketitle

\begin{abstract}
The first aim of this note is to fill a gap in the literature by proving that, given a global field $K$ and a finite set $\mathcal{S}$ of primes of $K$, every finite split embedding problem $G \rightarrow {\rm{Gal}}(L/K)$ over $K$ with nilpotent kernel has a solution ${\rm{Gal}}(F/K) \rightarrow G$ such that all primes in $\mathcal{S}$ are totally split in $F/L$. We then apply this to inverse Galois theory over division rings. Firstly, given a number field $K$ of level at least $4$, we show that every finite solvable group occurs as a Galois group over the division ring $H_K$ of quaternions with coefficients in $K$. Secondly, given a finite split embedding problem with nilpotent kernel over a finite field $K$, we fully describe for which automorphisms $\sigma$ of $K$ the embedding problem acquires a solution over the skew field of fractions $K(T, \sigma)$ of the twisted polynomial ring $K[T, \sigma]$.
\end{abstract}

\section{Introduction} \label{sec:intro}

The inverse Galois problem over a field $K$, a question which goes back to Hilbert and Noether, asks whether every finite group $G$ occurs as the Galois group of a Galois field extension $L/K$. By Shafarevich's theorem (see \cite[Theorem 9.6.1]{NSW08}), the answer to the latter question is affirmative if $K$ is a global field and $G$ is solvable. A refinement of the theorem, which is well-known to experts, is given by the following:

\begin{theorem} \label{thm:intro_0}
Let $K$ be a global field, $\mathcal{S}$ a finite set of primes of $K$, and $G$ a finite solvable group. There exists a Galois field extension $L/K$ of Galois group $G$ in which every prime in $\mathcal{S}$ is to\-tally split.
\end{theorem}

Theorem \ref{thm:intro_0} is stated as (part of) an exercise on the last page of \cite[Chapter IX]{NSW08}, with the hint that the totally split condition can be guaranteed by going through the proof of Shafarevich's theorem. However, no detailed solution is provided in \cite{NSW08}. We point out that special cases of Theorem \ref{thm:intro_0} were published in the literature after the first edition of \cite{NSW08} appeared. For example, Kl\"uners and Malle (see \cite[Theorem 6.1]{KM04}) assume $K$ is a number field and obtain the weaker conclusion that every prime in $\mathcal{S}$ is unramified in $L/K$. This was later improved by Checcoli and the first author (see \cite[Theorem 2.2 and Appendix A]{CF21}), who prove Theorem \ref{thm:intro_0} if $K$ is a number field. To our knowledge, no proof of Theorem \ref{thm:intro_0} is available in the literature. Our first aim is to explain how Theorem \ref{thm:intro_0} can be deduced from the literature (see \S\ref{ssec:main_1}).

To that end, we will prove the following theorem (see \S\ref{sec:proof}) about finite split embedding problems with nilpotent kernels over global fields. Given a field $K$, recall (see, e.g., \cite[\S16.4]{FJ08}) that a {\it{finite embedding problem over $K$}} is an epimorphism $\alpha : G \rightarrow {\rm{Gal}}(L/K)$, where $G$ is a finite group and $L/K$ a Galois field extension, and that $\alpha$ {\it{splits}} if there is an embedding $\alpha' : {\rm{Gal}}(L/K) \rightarrow G$ such that $\alpha \circ \alpha' = {\rm{id}}_{{\rm{Gal}}(L/K)}$. A {\it{solution to $\alpha$}} is an isomorphism $\beta : {\rm{Gal}}(F/K) \rightarrow G$, where $F$ is a Galois field extension of $K$ containing $L$, such that $\alpha \circ \beta$ is the restriction map ${\rm{Gal}}(F/K) \rightarrow {\rm{Gal}}(L/K)$.

\begin{theorem} \label{thm:intro_1}
Let $K$ be a global field, $\mathcal{S}$ a finite set of primes of $K$, and $\alpha : G \rightarrow {\rm{Gal}}(L/K)$ a finite embedding problem over $K$. Assume ${\rm{ker}}(\alpha)$ is nilpotent and $\alpha$ splits. Then there is a solution ${\rm{Gal}}(F/K) \rightarrow G$ to $\alpha$ such that every prime $\mathfrak{P} \in \mathcal{S}$ is totally split in $F/L$ (that is, every prime $\mathfrak{Q}$ of $L$ extending $\mathfrak{P}$ is totally split in $F/L$).
\end{theorem}

Theorem \ref{thm:intro_1} refines \cite[Theorem 9.6.6]{NSW08}, the main tool to prove Shafarevich's theorem, which asserts that every finite split embedding problem with nilpotent kernel over any given global field has a solution. 
Some special cases of Theorem \ref{thm:intro_1} are contained in the already mentioned works \cite{KM04} and \cite{CF21}, 
other special cases can be deduced from \cite[Theorem 14.3]{JR18} and \cite[Theorem B]{JR19}.

Our second aim is to contribute to inverse Galois theory over division rings. See \cite{DL20, ALP20, Beh21, BDL20f, Des20, Leg20} for some very recent results in this area. To state our results, we recall some definitions (see \S\ref{ssec:prelim_1} for more details). Firstly, for an automorphism $\sigma$ of a field $K$, we let $K[T, \sigma]$ be the ring of polynomials $a_0 + a_1 T + \cdots + a_n T^n$ with $n \geq 0$ and $a_0, \dots, a_n \in K$, whose addition is defined componentwise and multiplication fulfills $Ta = \sigma(a) T$ for $a \in K$. By $K(T, \sigma)$, we mean the unique division ring which contains $K[T, \sigma]$ and every element of which can be written as $ab^{-1}$ with $a \in K[T, \sigma]$ and $b \in K[T, \sigma] \setminus \{0\}$. If $\sigma={\rm{id}}_K$, we retrieve the usual commutative polynomial ring $K[T]$ and the rational function field $K(T)$, respectively. Secondly, recall that an extension $M/H$ of division rings is {\it{Galois}} (after Artin) if every element of $M$ which is fixed under every automorphism of $M$ fixing $H$ pointwise is in $H$. If $M/H$ is Galois, the automorphism group of $M/H$ is the {\it{Galois group}} ${\rm{Gal}}(M/H)$ of $M/H$.

Firstly, we combine Theorem \ref{thm:intro_0} and the main result from \cite{DL20} to get the following analogue of Shafarevich's theorem over division rings of quaternions. Recall that the {\it{level}} of a field $K$ is either the smallest positive integer $n$ such that there exist $x_1, \dots, x_n \in K$ with $-1 = x_1^2 + \cdots + x_n^2$ (if $-1$ can be written as the sum of finitely many squares in $K$), or $\infty$ (otherwise). See, e.g., \cite[Chapter XI, \S2]{Lam05} for more details.

\begin{theorem} \label{thm:intro_1.5}
Let $K$ be a number field of level at least 4 and $G$ a finite solvable group. Then $G$ occurs as the Galois group of a Galois extension of the division ring $H_K=K\oplus K\mathbf{i}\oplus K\mathbf{j}\oplus K\mathbf{k}$ ($\mathbf{i}^2= \mathbf{j}^2= \mathbf{k}^2 =\mathbf{i}\mathbf{j}\mathbf{k}=-1$) of quaternions with coefficients in $K$.
\end{theorem}

Secondly, we combine Theorem \ref{thm:intro_1} and results from \cite{BDL20f}, which extends the notion of finite embedding problems over fields to the situation of division rings of finite dimension over their centers. To state our result, note that every finite split embedding problem $\alpha : G \rightarrow {\rm{Gal}}(L/K)$ with nilpotent kernel over a finite field $K$ acquires a solution over the global field $K(T)$. That is, we compose $\alpha$ and the inverse of the restriction map ${\rm{Gal}}(L(T)/K(T)) \rightarrow {\rm{Gal}}(L/K)$, which is an isomorphism, to get a finite embedding pro\-blem $G \rightarrow {\rm{Gal}}(L(T)/K(T))$ over $K(T)$. The latter embedding problem splits and has nilpotent kernel, and so has a solution. The next theorem fully describes the automorphisms $\sigma$ of $K$ for which $\alpha$ acquires a solution over the division ring $K(T, \sigma)$.

\begin{theorem} \label{thm:intro_2}
Let $\alpha : G \rightarrow {\rm{Gal}}(L/K)$ be a finite split embedding problem with nilpotent kernel over a finite field $K$ and $\sigma$ an automorphism of $K$.
Then $\alpha$ acquires a solution  over $K(T,\sigma)$ 
if and only if the order of $\sigma$ is coprime to $[L:K]$.
\end{theorem}

\noindent
For a precise formulation of this, see Theorem~\ref{thm:app}, which is more general and relaxes the split assumption. The relevant definitions on finite embedding problems over division rings will be introduced in \S\ref{ssec:prelim_2}.

\section{Preliminaries} \label{sec:prelim}

We collect the material about division rings, finite embedding problems, and primes of global fields that will be used in the sequel.

\subsection{Division rings} \label{ssec:prelim_1}

In the following, a {\it{division ring}} is a non-zero (unital) ring in which all non-zero elements are invertible. Commutative division rings are nothing but {\it{fields}}.

Let $L/H$ be an extension (i.e., $H \subseteq L$) of division rings. The group of automorphisms of $L$ fixing $H$ pointwise is the {\it{automorphism group}} ${\rm{Aut}}(L/H)$ of $L/H$. Following Artin, we say that $L/H$ is {\it{Galois}} if every element of $L$ which is fixed under every element of ${\rm{Aut}}(L/H)$ is in $H$. If $L/H$ is Galois, ${\rm{Aut}}(L/H)$ is the {\it{Galois group}} ${\rm{Gal}}(L/H)$ of $L/H$.

A ring $R \not =\{0\}$ with no zero divisor is a {\it{right Ore domain}} if, for all $x, y \in R \setminus \{0\}$, there are $r, s \in R$ with $xr = ys \not=0$. If $R$ is a right Ore domain, there is a division ring $H$ which contains $R$ and every element of which can be written as $ab^{-1}$ with $a \in R$ and $b \in R \setminus \{0\}$ (see  \cite[Theo\-rem 6.8]{GW04}). Moreover, such a division ring $H$ is unique up to isomorphism (see \cite[Proposition 1.3.4]{Coh95}).

Let $H$ be a division ring and $\sigma$ an automorphism of $H$. The {\it{twisted polynomial ring}} $H[T, \sigma]$ is the ring of  polynomials $a_0 + a_1 T + \cdots + a_n T^n$ with $n \geq 0$ and $a_0, \dots, a_n \in H$, whose addition is defined componentwise and multiplication is given by 
$$\bigg(\sum_{i=0}^n a_i T^i \bigg) \cdot \bigg(\sum_{j=0}^m b_j T^j \bigg) = \sum_{k=0}^{n+m} \sum_{\ell=0}^k a_\ell \sigma^{\ell}(b_{k-\ell}) T^k.$$
Note that $H[T, \sigma]$ is commutative if and only if $H$ is a field and $\sigma={\rm{id}}_H$. In the sense of Ore (see \cite{Ore33}), $H[T, \sigma]$ is the twisted polynomial ring $H[T, \sigma, \delta]$ in the variable $T$, where the derivation $\delta$ is 0. The ring $H[T, \sigma]$ has no zero divisor, as the degree is additive on products, and is a right Ore domain (see \cite[Theorem 2.6 and Corollary 6.7]{GW04}). The unique division ring which contains $H[T, \sigma]$ and each element of which can be written as $ab^{-1}$ with $a \in H[T, \sigma]$ and $b \in H[T, \sigma] \setminus \{0\}$ is then denoted $H(T, \sigma)$. If $\sigma={\rm{id}}_H$, we write $H[T]$ and $H(T)$ instead of $H[T, {\rm{id}}_H]$ and $H(T, {\rm{id}}_H)$, respectively. If $H$ is a field, $H(T)$ is nothing but the usual field of fractions of the commutative polynomial ring $H[T]$.

\subsection{Finite embedding problems} \label{ssec:prelim_2}

First, let $L/H$ and $F/M$ be two Galois extensions of division rings with finite Galois groups, and such that $L \subseteq F$ and $H \subseteq M$. We write
$${\rm{res}}^{F/M}_{L/H}$$
for the restriction map ${\rm{Gal}}(F/M) \rightarrow {\rm{Gal}}(L/H)$ (that is, ${\rm{res}}^{F/M}_{L/H}(\sigma)(x)=\sigma(x)$ for every $\sigma \in {\rm{Gal}}(F/M)$ and every $x \in L$), if it is well-defined.

Unlike the commutative case, ${\rm{res}}^{F/M}_{L/H}$ is not always well-defined. The next result (see the special case III) of \cite[\S3.1]{BDL20f}) gives a practical situation where it is well-defined:

\begin{proposition} \label{prop:restriction}
Let $H$ be a division ring of finite dimension over its center. Let $L/H$ and $F/H$ be two Galois extensions of division rings with finite Galois groups and such that $L \subseteq F$. Then the restriction map ${\rm{res}}^{F/H}_{L/H}$ is well-defined.
\end{proposition}

Now, let $H$ be a division ring of finite dimension over its center. A {\it{finite embedding problem}} over $H$ is an epimorphism $\alpha : G \rightarrow {\rm{Gal}}(L/H)$, where $G$ and $L/H$ are a finite group and a Galois extension of division rings, respectively. We say that $\alpha$ {\it{splits}} if there is an embedding $\alpha' : {\rm{Gal}}(L/H) \rightarrow G$ with $\alpha \circ \alpha' = {\rm{id}}_{{\rm{Gal}}(L/H)}$. A {\it{weak solution}} to $\alpha$ is a monomorphism $\beta : {\rm{Gal}}(F/H) \rightarrow G$, where $F/H$ is a Galois extension of division rings with $L \subseteq F$, such that $\alpha \circ \beta$ is the restriction map ${\rm{res}}^{F/H}_{L/H}$ (which is well-defined by Proposition \ref{prop:restriction}). If $\beta$ is an isomorphism, we say {\it{solution}} instead of weak solution.

\begin{remark} \label{rk:fep}
Let $L/H$ be a Galois extension of division rings with ${\rm{Gal}}(L/H)$ finite. Then $H$ is a field if and only if $L$ is (see \cite[lemme 2.1 and th\'eor\`eme 2.2]{BDL20f}). Hence, the above terminology generalizes that of the commutative case (see \S\ref{sec:intro}).
\end{remark}

Finally, let $H$ be a division ring of finite dimension over its center and $\sigma$ an automorphism of $H$ of finite order. Let $\alpha : G \rightarrow {\rm{Gal}}(L/H)$ be a finite embedding problem over $H$ and $\tau$ an automorphism of $L$ of finite order exten\-ding $\sigma$. Assume this condition holds:
{\renewcommand{\theequation}{2.1}
\begin{equation}\label{condL}
\begin{minipage}{11cm}
{\it{$L(T, \tau) / H(T, \sigma)$ is Galois with finite Galois group, and the restriction map ${\rm{res}}^{L(T,\tau) / H(T, \sigma)}_{L/H}$ exists and is an isomorphism.}}
\end{minipage}
\end{equation}
Then
{\renewcommand{\theequation}{2.2}
\begin{equation} \label{eq:st}
\alpha_{\sigma, \tau} = ({\rm{res}}^{L(T, \tau)/H(T, \sigma)}_{L/H})^{-1} \circ \alpha : G \rightarrow {\rm{Gal}}(L(T, \tau)/H(T, \sigma))
\end{equation}
is a finite embedding problem over $H(T, \sigma)$, which is of finite dimension over its center (see \cite[lemme 2.3]{BDL20f}). A {\it{$(\sigma, \tau)$-geometric solution}} to $\alpha$ is a solution to $\alpha_{\sigma, \tau}$. If $\tau={\rm{id}}_L$ (and so $\sigma={\rm{id}}_H$), we say {\it{geometric solution}} for simplicity. By Remark \ref{rk:fep}, if $H$ is a field and ${\rm{Gal}}(E/H(T)) \rightarrow G$ a geometric solution to $\alpha$, then $E$ is a field.

\subsection{Primes of global fields} \label{ssec:prelim_3}

Recall that a field $K$ is {\it{global}} if $K$ is either a number field or a finitely generated field extension of a finite field with transcendence degree 1. If $K$ is a global field of characteristic $p > 0$, there is a transcendental $T$ such that $K$ is a finite separable extension of $\mathbb{F}_p(T)$.

Let $K$ be a global field. A {\it{prime}} of $K$ is an equivalence class of non-trivial absolute values on $K$. If $K$ is a number field, {\it{non-archimedean primes}} of $K$ are in 1-to-1 correspondence with maximal ideals of the ring of integers of $K$, and {\it{archimedean primes}} of $K$ are equivalence classes of non-trivial absolute values on $K$ whose restriction to $\mathbb{Q}$  is equivalent to the ``usual" absolute value. Now, if $K$ is global of characteristic $p > 0$, every prime of $K$ is {\it{non-archimedean}}. If $T$ is a transcendental as above, the set of  primes of $K$ is in bijection with the set $\mathfrak{S}_1 \cup \mathfrak{S}_2$, where $\mathfrak{S}_1$ is the set of maximal ideals of the integral closure of $\mathbb{F}_p[T]$ in $K$, and $\mathfrak{S}_2$ is the set of maximal ideals of the integral closure of $\mathbb{F}_p[1/T]$ in $K$ containing $1/T$.

For a prime $\mathfrak{P}$ of a global field $K$, we let $K_\mathfrak{P}$ denote the completion of $K$ at $\mathfrak{P}$. If $L/K$ is a Galois extension of global fields, we say that a prime $\mathfrak{P}$ of $K$ is {\it{totally split in $L/K$}} if $K_\mathfrak{P}$ equals the completion $L_{\mathfrak{P}'}$ of $L$ at any prime $\mathfrak{P}'$ of $L$ extending $\mathfrak{P}$. If $\mathfrak{P}$ is non-archimedean, then $\mathfrak{P}$ is totally split in $L/K$ if and only if both the ramification index and the residue degree of $L/K$ at (the maximal ideal corresponding to) $\mathfrak{P}$ equal 1. 

If $K \subseteq L \subseteq F$ are global fields such that $F/K$ and $L/K$ are Galois, and if $\mathfrak{P}$ is a prime of $K$, we say that $\mathfrak{P}$ is {\it{totally split in $F/L$}} if any prime $\mathfrak{Q}$ of $L$ extending $\mathfrak{P}$ is totally split in $F/L$. We also say that the completion of $L$ at $\mathfrak{Q}$ is the {\it{completion of $L$ at $\mathfrak{P}$}}. If $\mathfrak{P}$ is non-archimedean, the ramification index of $F/L$ at $\mathfrak{Q}$ and the residue field of $L$ at $\mathfrak{Q}$ are the {\it{ramification index of $F/L$ at $\mathfrak{P}$}} and the {\it{residue field of $L$ at $\mathfrak{P}$}}, respectively.

\section{Proofs of Theorems \ref{thm:intro_0}, \ref{thm:intro_1.5}, and \ref{thm:intro_2} under Theorem \ref{thm:intro_1}} \label{sec:main}

\subsection{Proof of Theorem \ref{thm:intro_0}} \label{ssec:main_1}

We proceed, as in the proof of Shafarevich's theorem given right after \cite[Proposition 9.6.9]{NSW08}, by induction on $|G|$. Suppose Theorem \ref{thm:intro_0} holds for any finite solvable group of order less than $|G|$. By \cite[Propositions 9.6.8 and 9.6.9]{NSW08}, there is a surjection $\varphi : N \rtimes G' \rightarrow G$, where $N$ is the (nilpotent) Fitting subgroup of $G$ and $G'$ is a proper subgroup of $G$. By the induction hypothesis, there is a Galois field extension $L/K$ of group $G'$ in which all primes in $\mathcal{S}$ are totally split. Let $\gamma : G' \rightarrow {\rm{Gal}}(L/K)$ be an isomorphism and ${\rm{pr}} : N \rtimes G' \rightarrow G'$ the projection on the second coordinate. Consider the finite embedding problem $\gamma \circ {\rm{pr}} : N \rtimes G' \rightarrow {\rm{Gal}}(L/K)$ over $K$; it splits and has nilpotent kernel $N\times \{1\}$. We may then apply Theorem \ref{thm:intro_1} to get the existence of a solution ${\rm{Gal}}(F/K) \rightarrow N \rtimes G'$ to $\gamma \circ {\rm{pr}}$ such that all primes in $\mathcal{S}$ are totally split in $F/L$. As the same holds in $L/K$, all primes in $\mathcal{S}$ are totally split in $F/K$. Then $F^{{\rm{ker}}(\varphi)}/K$ is a Galois field extension of group $G$, in which all primes in $\mathcal{S}$ are totally split.

\subsection{Proof of Theorem \ref{thm:intro_1.5}} \label{ssec:main_1.5}

As the number field $K$ has level at least 4, we may apply the Hasse--Minkowski theorem (see, e.g., \cite[p. 170]{Lam05}) to get the existence of a prime $\mathfrak{P}$ of $K$ such that the completion $K_\frak{P}$ of $K$ at $\frak{P}$ has level at least 4. By Theorem \ref{thm:intro_0}, there exists a Galois field extension $L/K$ of group $G$ such that $L \subseteq K_\mathfrak{P}$. In particular, $L$ has level at least 4. It then remains to apply \cite[th\'eor\`eme 7]{DL20} to conclude that the division ring $H_L$ of quaternions with coefficients in $L$ is a Galois extension of $H_K$ with Galois group $G$.

\subsection{Proof of Theorem \ref{thm:intro_2}} \label{ssec:main_2}

It is well-known that if $\alpha$ is a finite embedding problem with nilpotent kernel over a finite field $K$, then $\alpha$ has a geometric solution. Indeed, by the projectivity of the absolute Galois group of the finite field $K$ (see, e.g., \cite[Proposition 11.6.6]{FJ08} and \cite[Proposition 6.1.3]{GS17}), $\alpha$ has a weak solution. The existence of the latter and the {\it{weak$\rightarrow$split reduction}} (see \cite[\S1 B) 2)]{Pop96}) then provide a finite split embedding problem $\alpha'$ over $K$ which fulfills the following two properties:

\noindent
(i) ${\rm{ker}}(\alpha') \cong {\rm{ker}}(\alpha)$,

\noindent
(ii) if $\alpha'$ has a geometric solution, then $\alpha$ has a geometric solution.

\noindent
By (i) and the assumption that ${\rm{ker}}(\alpha)$ is nilpotent, ${\rm{ker}}(\alpha')$ is nilpotent. Hence, by \cite[Theorem 9.6.6]{NSW08}, the finite split embedding problem $\alpha'$ over the finite field $K$ has a geometric solution. It then remains to use (ii) to get that $\alpha$ has a geometric solution, as claimed.

We now provide the same conclusion over more division rings of the form $K(T, \sigma)$, where $\sigma$ is an automorphism of $K$. The next theorem generalizes Theorem \ref{thm:intro_2}.

\begin{theorem} \label{thm:app}
Let $\alpha : G \rightarrow {\rm{Gal}}(L/K)$ be a finite embedding problem with nilpotent kernel over a finite field $K$, let $\sigma \in {\rm{Aut}}(K)$, and let $d$ be the order of $\sigma$. Consider these three conditions:

\vspace{0.5mm}

\noindent
{\rm{(a)}} $\alpha$ has a weak solution $\gamma : {\rm{Gal}}(L'/K) \rightarrow G$ such that $d$ and $[L':K]$ are coprime,

\vspace{0.5mm}

\noindent
{\rm{(b)}} there exists $\tau \in {\rm{Aut}}(L)$ extending $\sigma$ such that $\alpha$ has a $(\sigma, \tau)$-geometric solution,

\vspace{0.5mm}

\noindent
{\rm{(c)}} $d$ and $[L:K]$ are coprime.

\vspace{0.5mm}

\noindent
Then we have the following four conclusions:

\vspace{0.5mm}

\noindent
{\rm{(1)}} {\rm{(a)}} $\Rightarrow$ {\rm{(b)}} $\Rightarrow$ {\rm{(c)}},

\vspace{0.5mm}

\noindent
{\rm{(2)}} if $\alpha$ splits, then {\rm{(a)}} $\Leftrightarrow$ {\rm{(b)}} $\Leftrightarrow$ {\rm{(c)}},

\vspace{0.5mm}

\noindent
{\rm{(3)}} if {\rm{(a)}} holds, then an automorphism $\tau$ of $L$ as in {\rm{(b)}} is unique,

\vspace{0.5mm}

\noindent
{\rm{(4)}} if {\rm{(c)}} fails, then \eqref{condL} fails for every $\tau \in {\rm{Aut}}(L)$ extending $\sigma$.
\end{theorem}

Note that the existence of some weak solution to $\alpha$ is automatic from the projectivity of the absolute Galois group of the finite field $K$. 

As defined in \S\ref{ssec:prelim_2}, a $(\sigma, \tau)$-geometric solution to a finite embedding problem $\alpha$ over a division ring $H$ of finite dimension over its center is a solution to the finite embedding problem $\alpha_{\sigma, \tau}$ over $H(T, \sigma)$, which is introduced in \eqref{eq:st}. To make sure that $\alpha_{\sigma, \tau}$ is well-defined, we assumed \eqref{condL}. In the next lemma, of which Condition (1) is nothing but \eqref{condL}, we make \eqref{condL} explicit, if $H$ is a finite field.

\begin{lemma} \label{lemma_1}
Let $L/K$ be an extension of finite fields, $\sigma \in {\rm{Aut}}(K)$, and $\tau \in {\rm{Aut}}(L)$ extending $\sigma$. Let $d$ denote the order of $\sigma$. The following three conditions are equivalent:

\vspace{0.5mm}

\noindent
{\rm{(1)}} $L(T, \tau) / K(T, \sigma)$ is Galois with finite Galois group, and the restriction map ${\rm{res}}^{L(T,\tau) / K(T, \sigma)}_{L/K}$ exists and is an isomorphism,

\vspace{0.5mm}

\noindent
{\rm{(2)}} $\tau$ has order $d$, and $d$ and $[L:K]$ are coprime,

\vspace{0.5mm}

\noindent
{\rm{(3)}} $\tau$ has order $d$, and the subgroup $\langle \tau, {\rm{Gal}}(L/K) \rangle$ of ${\rm{Aut}}(L)$ equals $\langle \tau \rangle \times {\rm{Gal}}(L/K)$.
\end{lemma}

\begin{proof}
The equivalence (1) $\Leftrightarrow$ (3) is a special case of \cite[corollaire 3.4 and proposition 3.8]{BDL20f}. It then suffices to show that (2) and (3) are equivalent. To that end, note that $\langle \tau \rangle$ and ${\rm{Gal}}(L/K)$ are subgroups of the cyclic group ${\rm Aut}(L)$. Hence, $\langle \tau, {\rm{Gal}}(L/K) \rangle = \langle \tau \rangle \times {\rm{Gal}}(L/K)$ if and only if the order of $\tau$ and $[L:K]$ are coprime, thus showing (2) $\Leftrightarrow$ (3).
\end{proof}

\begin{proof}[Proof of Theorem \ref{thm:app}]
We first prove (1) and (3) simultaneously. Since (b) $\Rightarrow$ (c) follows from (1) $\Rightarrow$ (2) in Lemma \ref{lemma_1}, it suffices to prove (a) $\Rightarrow$ (b) and the uniqueness of $\tau$ under (a). To that end, let $\gamma : {\rm{Gal}}(L'/K) \rightarrow G$ be a weak solution to $\alpha$ such that $d$ and $[L':K]$ are coprime. In particular, ${\rm{gcd}}(d, [L:K])=1$ and, consequently, there is $\tau \in {\rm{Aut}}(L)$ of order $d$ extending $\sigma$, and $\tau$ is necessarily uni\-que. From (2) $\Leftrightarrow$ (3) in Lemma \ref{lemma_1}, we get $\langle \tau, {\rm{Gal}}(L/K) \rangle = \langle \tau \rangle \times {\rm{Gal}}(L/K)$. Similarly, there is a unique $\tau' \in {\rm{Aut}}(L')$ of order $d$ extending $\sigma$, which actually extends $\tau$, and, from (2) $\Leftrightarrow$ (3) in Lemma \ref{lemma_1}, we get $\langle \tau', {\rm{Gal}}(L'/K) \rangle = \langle \tau' \rangle \times {\rm{Gal}}(L'/K).$ We may then apply the weak$\rightarrow$split reduction for finite embedding problems over division rings \cite[proposition 5.3]{BDL20f} to get the existence of a finite split embedding problem $\alpha' : G' \rightarrow {\rm{Gal}}(L'/K)$ over $K$ fulfilling the following two properties:

\noindent
(i) ${\rm{ker}}(\alpha') \cong {\rm{ker}}(\alpha)$, 

\noindent
(ii) if $\alpha'$ has a $(\sigma, \tau')$-geometric solution, then $\alpha$ has a $(\sigma, \tau)$-geometric solution.

Now, let $K^{\langle \sigma \rangle}$ (resp., $L'^{\langle \tau' \rangle}$) be the fixed field of $\langle \sigma \rangle$ (resp., of $\langle \tau' \rangle$) in $K$ (resp., in $L'$). As $\langle \tau', {\rm{Gal}}(L'/K) \rangle = \langle \tau' \rangle \times {\rm{Gal}}(L'/K)$ (see the previous paragraph), we may apply \cite[lemme 3.5]{BDL20f} to get that $L'^{\langle \tau' \rangle}/K^{\langle \sigma \rangle}$ is Galois. Moreover, as the orders of $\sigma$ and $\tau'$ are equal, we may apply \cite[lemme 2.4]{BDL20f} to get that $L'^{\langle \tau' \rangle}$ and $K$ are linearly disjoint over $K^{\langle \sigma \rangle}$, and that $L' = L'^{\langle \tau' \rangle} K$. Therefore, ${\rm{res}}^{L'/K}_{L'^{\langle \tau' \rangle}/K^{\langle \sigma \rangle}}$ is an isomorphism. Hence, $${\overline{\alpha'}}_{\sigma, \tau'}  = {\rm{res}}^{L'/K}_{L'^{\langle \tau' \rangle}/K^{\langle \sigma \rangle}} \circ \alpha' : G' \rightarrow {\rm{Gal}}(L'^{\langle \tau' \rangle}/K^{\langle \sigma \rangle})$$ 
is a finite embedding problem over $K^{\langle \sigma \rangle}$, which splits and has nilpotent kernel (by (i) and the assumption on ${\rm{ker}}(\alpha)$). Theorem \ref{thm:intro_1} then yields that ${\overline{\alpha'}}_{\sigma, \tau'}$ has a geometric solution ${\rm{Gal}}(F'/K^{\langle \sigma \rangle}(T)) \rightarrow G'$ such that $F' \subseteq L'^{\langle \tau' \rangle}((T))$. Hence, by \cite[lemme 4.2]{BDL20f}, $\alpha'$ has a $(\sigma, \tau')$-geometric solution. It then remains to apply (ii) to conclude.

Now, we prove (2). To that end, assume $\alpha$ splits. By (1), it suffices to prove (c) $\Rightarrow$ (a). As $\alpha$ splits, there is an embedding $\alpha' : {\rm{Gal}}(L/K) \rightarrow G$ such that $\alpha \circ \alpha' = {\rm{id}}_{{\rm{Gal}}(L/K)}$. Then $\alpha'$ is a weak solution to $\alpha$ and, if (c) holds, then (a) holds with $\gamma = \alpha'$.

Finally, we prove (4). If (c) fails, then Condition (2) from Lemma \ref{lemma_1} fails too. Then, from (1) $\Leftrightarrow$ (2) in Lemma \ref{lemma_1}, we get that \eqref{condL} also fails.
\end{proof}

\section{Proof of Theorem \ref{thm:intro_1}} \label{sec:proof}

Finally, we proceed to the proof of Theorem \ref{thm:intro_1}. For the convenience of the reader, we restate the theorem here:

\begin{theorem} \label{thm:main}
Let $K$ be a global field, $\mathcal{S}$ a finite set of primes of $K$, and $\alpha : G \rightarrow {\rm{Gal}}(L/K)$ a finite embedding problem over $K$. Assume ${\rm{ker}}(\alpha)$ is nilpotent and $\alpha$ splits. Then there exists a solution ${\rm{Gal}}(F/K) \rightarrow G$ to $\alpha$ such that every prime $\mathfrak{P} \in \mathcal{S}$ is totally split in $F/L$.
\end{theorem}

The structure of the proof is similar to that of \cite[Theorem 9.6.6]{NSW08}. Namely, we first reduce Theorem \ref{thm:main} to the case of finite split embedding problems whose kernels are certain $p$-groups (see \S\ref{ssec:proof_1}). The latter case is then proved in two steps, depending on whether $p$ equals the characteristic of $K$ (see \S\ref{ssec:proof_2} and \S\ref{ssec:proof_3}).

\subsection{General reduction} \label{ssec:proof_1}

For a prime number $p$ and an integer $n \geq 1$, let $\mathcal{F}_p(n)$ be the free pro-$p$-${\rm{Gal}}(L/K)$ operator group of rank $n$ as defined before \cite[Proposition 9.6.3]{NSW08}. For $\nu=(i,j)$ with $i \geq j \geq 1$, we let $\mathcal{F}_p(n)^{(\nu)}$ denote the filtration of $\mathcal{F}_p(n)$ refining the descending $p$-central series as in \cite[Definition 3.8.7]{NSW08}. Since every finite nilpotent group is a direct product of its Sylow subgroups, and each finite ${\rm{Gal}}(L/K)$-operator $p$-group is a quotient of $\mathcal{F}_p(n)/\mathcal{F}_p(n)^{(\nu)}$ for some $n$ and $\nu$ (see right after \cite[Theorem 9.6.6]{NSW08}), Theorem \ref{thm:intro_1} reduces to proving the following statement, which partially refines \cite[Theorem 9.6.7]{NSW08}, for every prime number $p$:
{\renewcommand{\theequation}{$4.1$}
\begin{equation}\label{dagger}
\begin{minipage}{14cm}
{\it{For each integer $n \geq 1$ and each $\nu=(i,j)$, the finite split embedding problem 
$${\rm{pr}} : \mathcal{F}_p(n) / \mathcal{F}_p(n)^{(\nu)} \rtimes {\rm{Gal}}(L/K) \rightarrow {\rm{Gal}}(L/K)$$ 
over the field $K$, given by the projection on the second coordinate, has a solution 
$$\gamma : {\rm{Gal}}(F/K) \rightarrow \mathcal{F}_p(n) / \mathcal{F}_p(n)^{(\nu)} \rtimes {\rm{Gal}}(L/K)$$ 
such that every prime $\mathfrak{P} \in \mathcal{S}$ is totally split in $F/L$.}}
\end{minipage}
\end{equation}

We break the proof into two parts. Let $p_0 \geq 0$ be the characteristic of $K$.

\subsection{The case $p \not=p_0$} \label{ssec:proof_2}

First, assume $p \not=p_0$. If all non-archimedean primes in $\mathcal{S}$ ramify in $L/K$, then \eqref{dagger} follows from \cite[Theorem 9.6.7(i)]{NSW08}. To reduce to this case, we replace $L$ by the compositum $LL'$ of $L$ and some finite Galois field extension $L'$ of $K$ which is linearly disjoint from $L$ over $K$, and which has specified local behaviour at primes $\mathfrak{P} \in \mathcal{S}$.

\begin{lemma} \label{lemma:L'}
There is a finite Galois field extension $L'$ of $K$ which is linearly disjoint from $L$ over $K$, and which satisfies the following for every prime $\mathfrak{P}\in\mathcal{S}$:

\vspace{0.5mm}

\noindent
{\rm{(1)}} if $\mathfrak{P}$ is non-archimedean and unramified in $L/K$, then the completion at $\mathfrak{P}$ of $L'/K$ ramifies and its degree is not divisible by $p$,

\vspace{0.5mm}

\noindent
{\rm{(2)}} if $\mathfrak{P}$ is either archimedean or non-archimedean and ramified in $L/K$, then $\mathfrak{P}$ is totally split in $L'/K$.
\end{lemma}

\begin{proof}
First, write $\mathcal{S} = \{\mathfrak{P}_1, \dots, \mathfrak{P}_r\}$. For $i = 1, \dots, r$, we let $F_{i}$ denote the following Galois field extension of $K_{\mathfrak{P}_i}$:
\begin{enumerate}[ (a)]
\item $F_{i}=K_{\mathfrak{P}_i}$ if $\mathfrak{P}_i$ is either archimedean or non-archimedean and ramified in $L/K$,
\item $F_{i}$ is a ramified quadratic field extension of $K_{\mathfrak{P}_i}$, if $p \not=2$ and $\mathfrak{P}_i$ is non-archimedean and unramified in $L/K$,
\item $F_{i}$ is a ramified finite Galois field extension of $K_{\mathfrak{P}_i}$ of odd degree, if $p=2$ and $\mathfrak{P}_i$ is non-archimedean and unramified in $L/K$.
\end{enumerate}
We briefly explain why an extension $F_i$ as in (c) exists: If $q$ is the cardinality of the residue field of $K$ at $\mathfrak{P}_i$,
then as $q^3-1 = (q-1)(q^2+q+1)$, there is some odd prime number $p'$ with $q^3 \equiv 1 \mbox{ mod }p'$. The latter congruence is a sufficient condition for the existence of a Galois extension $F_i$ of $K_{\mathfrak{P}_i}$ with ramification index $p'$ and residue degree $3$, see \cite[pp. 253-254]{Has80}. In particular, $F_i/K_{\mathfrak{P}_i}$ ramifies and $[F_i: K_{\mathfrak{P}_i}]= 3 p'$ is odd.

We also let $\mathfrak{P}_{r+1}$ be a prime of $K$ not in $\{\mathfrak{P}_1, \dots, \mathfrak{P}_r\}$ that is non-archimedean and unramified in $L/K$, and choose a ramified quadratic field extension $F_{{r+1}}$ of $K_{\mathfrak{P}_{r+1}}$. Moreover, let $n$ be an integer with $n \geq [F_i : K_{\mathfrak{P}_i}]$ for $i= 1 , \dots, r+1$. Let $\mathfrak{P}_{r+2}, \mathfrak{P}_{r+3}, \mathfrak{P}_{r+4}$ be distinct non-archimedean primes of $K$ not in $\{ \mathfrak{P}_{1} , \dots, \mathfrak{P}_{r+1}\}$. For $i=r+2, r+3, r+4$, let $F_i$ be the unramified Galois field extension of $K_{\mathfrak{P}_i}$ of degree $n_i$, where $(n_{r+2}, n_{r+3}, n_{r+4}) = (n, n-1, 2)$.

Now, for $i= 1 , \dots, r+4$, let $P_i(X) \in K_{\frak{P}_i}[X]$ be the minimal polynomial of a primitive element of $F_i$ over $K_{\frak{P}_i}$, and let $Q_i(X) \in K_{\frak{P}_i}[X]$ be a monic separable polynomial of degree $n$ which is the product of $P_i(X)$ and $n - [F_i : K_{\mathfrak{P}_i}]$ monic degree 1 polynomials with coefficients in $K_{\frak{P}_i}$. By the weak approximation theorem (see, e.g., \cite[Chapter XII, Theorem 1.2]{Lan02}) and Krasner's lemma (e.g., in the form of \cite[Proposition 12.3]{Jar91}), there exists a monic separable polynomial $Q(X) \in K[X]$ of degree $n$ which fulfills this property: 
\renewcommand{\theequation}{$4.2$}
\begin{equation}\label{dagger2}
\begin{minipage}{14.5cm}
{\it{For $i=1, \dots, r+4$, if $x_1, \dots, x_n$ are the roots of $Q(X)$, then the roots of $Q_i(X)$ can be enumerated as $y_{i,1}, \dots, y_{i,n}$ such that $K_{\frak{P}_i}(x_j) = K_{\frak{P}_i}(y_{i,j})$ for $j=1, \dots, n$.}}
\end{minipage}
\end{equation}
In particular, the splitting field $L'$ of $Q(X)$ over $K$ satisfies $L' K_{\frak{P}_i} = F_i$ for $i=1, \dots, r+4$. From the definition of $F_i$ for $i = 1, \dots, r$, we get that $L'/K$ fulfills (1) and (2) in the statement of the lemma.

Finally, we show the remaining claim that $L$ and $L'$ are linearly disjoint over $K$. For $i= r+2, r+3, r+4$, the definition of $Q_i(X)$ and \eqref{dagger2} yield that the Galois group $G_i$ of $Q(X)$ over $K_{\frak{P}_i}$ acts on $x_1, \dots, x_n$ as an $n_i$-cycle. Since $(n_{r+2}, n_{r+3}, n_{r+4}) = (n, n-1, 2)$ and $G_{r+2}, G_{r+3}, G_{r+4}$ are subgroups of the Galois group $G$ of $Q(X)$ over $K$, we get that $G$ contains an $n$-cycle, an $(n-1)$-cycle, and a transposition. Hence, $G=S_n$ (see, e.g., \cite[Lemma 4.4.3]{Ser92}). Moreover, we get similarly that the Galois group of $Q(X)$ over $K_{\frak{P}_{r+1}}$ acts on $x_1, \dots, x_n$ as a transposition, in particular as an odd permutation. Since $F_{r+1}/K_{\frak{P}_{r+1}}$ ramifies, we obtain that $\mathfrak{P}_{r+1}$ ramifies already in the quadratic subfield $L''=L'^{A_n}$ of $L'$ and, as $\mathfrak{P}_{r+1}$ is unramified in $L/K$, this implies that $L\cap L''=K$. As every proper normal subgroup of $S_{n}$ is contained in $A_{n}$, we eventually get that $L \cap L' = K$, as needed.
\end{proof}

\begin{remark}
In the case $p \not=2$, the proof shows that $L'/K$ may be chosen to be quadratic.
\end{remark}

\begin{proof}[Proof of \eqref{dagger} in the case $p \not=p_0$]
Let $n \geq 1$ be an integer and $\nu = (i,j)$. Consider the finite split embedding problem ${\rm{pr}} : \mathcal{F}_p(n)/\mathcal{F}_p(n)^{(\nu)} \rtimes {\rm{Gal}}(L/K) \rightarrow {\rm{Gal}}(L/K)$ over $K$, given by the projection on the second coordinate. Let $L'/K$ be as in Lemma \ref{lemma:L'}.

Since $L$ and $L'$ are linearly disjoint over $K$, the map 
$${\rm{res}} : \left \{ \begin{array} {ccc} $$
{\rm{Gal}}(LL'/K) & \rightarrow & {\rm{Gal}}(L/K) \times {\rm{Gal}}(L'/K) \\
\sigma & \mapsto & ({\rm{res}}^{LL'/K}_{L/K}(\sigma), {\rm{res}}^{LL'/K}_{L'/K}(\sigma))
\end{array} \right.$$
is an isomorphism. Then consider the finite embedding problem 
$$\alpha = {\rm{res}}^{-1} \circ ({\rm{pr}} \times {\rm{id}}_{{\rm{Gal}}(L'/K)}) : \left \{ \begin{array} {ccc}
( \mathcal{F}_p(n) / \mathcal{F}_p(n)^{(\nu)} \rtimes {\rm{Gal}}(L/K)) \times {\rm{Gal}}(L'/K) & \rightarrow & {\rm{Gal}}(LL'/K) \\
((x,y),z) & \mapsto & {\rm{res}}^{-1}(y,z)
\end{array} \right. $$ 
over $K$; it splits and has nilpotent kernel. As all non-archimedean primes in $\mathcal{S}$ are ra\-mi\-fied in $LL'/K$ and $p \not=p_0$, \cite[Theorem 9.6.7(i)]{NSW08} gives a solution
$$\beta : {\rm{Gal}}(F/K) \rightarrow (\mathcal{F}_p(n) / \mathcal{F}_p(n)^{(\nu)} \rtimes {\rm{Gal}}(L/K)) \times {\rm{Gal}}(L'/K)$$ 
to $\alpha$ such that every prime $\mathfrak{P} \in \mathcal{S}$ is totally split in $F/LL'$. Set 
$$M=F^{\beta^{-1}((\{1\} \times \{1\}) \times {\rm{Gal}}(L'/K))}.$$ 
Then $L \subseteq M$ and $\beta$ induces a solution ${\rm{Gal}}(M/K) \rightarrow \mathcal{F}_p(n) / \mathcal{F}_p(n)^{(\nu)} \rtimes {\rm{Gal}}(L/K)$ to ${\rm{pr}}$.

It remains to show that every prime $\mathfrak{P} \in \mathcal{S}$ is totally split in $M/L$. First, assume $\mathfrak{P}$ is non-archimedean and ramified in $L/K$. Then $\mathfrak{P}$ is unramified in $L'/K$. Hence, the ramification index at $\mathfrak{P}$ of $LL'/L$ is 1. As $\mathfrak{P}$ is totally split in $F/LL'$, we get that the ramification index at $\mathfrak{P}$ of $F/L$ is 1, and so the same holds for $M/L$. Moreover, denoting residue fields at $\mathfrak{P}$ by $\overline{\bullet}$, we have $\overline{F}= \overline{LL'}=\overline{L}\cdot\overline{L'}$ (see \cite[Lemma 2.4.8]{FJ08} for the last equality). Since $\overline{L'}=\overline{K}$, we get that $\overline{F} = \overline{L}$, and so $\overline{M} = \overline{L}$.

Now, assume $\mathfrak{P}$ is non-archimedean and unramified in $L/K$. Then $p$ does not divide the ramification index at $\mathfrak{P}$ of $L'/K$, and so does not divide that of $LL'/L$ either. As $\mathfrak{P}$ is totally split in $F/LL'$, we get that the ramification index at $\mathfrak{P}$ of $F/L$ is not divisible by $p$. Since $[M:L]$ is a power of $p$, the ramification index at $\mathfrak{P}$ of $M/L$ is then 1. The argument is similar for residue fields. Namely, with the notation from above, $[\overline{L'}:\overline{K}]$ and $p$ are coprime, and hence the same holds for $[\overline{LL'}:\overline{L}]$ and $p$. As $\mathfrak{P}$ is totally split in $F/LL'$, this implies that $p$ does not divide $[\overline{F}:\overline{L}]$, and so $p$ does not divide  $[\overline{M}: \overline{L}]$ either. As $[\overline{M}:\overline{L}]$ is a power of $p$, we get $\overline{M}=\overline{L}$.

Finally, assume $\mathfrak{P} \in \mathcal{S}$ is archimedean and $L_\mathfrak{P}=\mathbb{R}$. By the definition of $L'$, we have $(L')_{\mathfrak{P}}=K_\mathfrak{P}=\mathbb{R}$, hence $(LL')_\mathfrak{P}=\mathbb{R}$. Since $\mathfrak{P}$ is totally split in $F/LL'$, we get that $F_\mathfrak{P}=\mathbb{R}$. In particular, $M_\mathfrak{P}=\mathbb{R}$.
\end{proof}

\subsection{The case $p=p_0$} \label{ssec:proof_3}

Now, assume $p=p_0$. Given $n$ and $\nu$, consider the embedding
$$\alpha' : \left \{ \begin{array} {ccc} $$
{\rm{Gal}}(L/K) & \rightarrow &  \mathcal{F}_p(n) / \mathcal{F}_p(n)^{(\nu)} \rtimes {\rm{Gal}}(L/K) \\
\sigma & \mapsto & (1, \sigma)
\end{array} \right..$$
For a prime $\mathfrak{P}$ of $K$, set 
$$\psi_\mathfrak{P} = \alpha' \circ {\rm{res}}^{K^{\rm{sep}}/K}_{L/K} \circ {\rm{res}}_{K^{\rm{sep}}/K}^{(K_\mathfrak{P})^{\rm{sep}}/K_\mathfrak{P}} : {\rm{Gal}}((K_\mathfrak{P})^{\rm{sep}}/K_\mathfrak{P}) \rightarrow \mathcal{F}_p(n) / \mathcal{F}_p(n)^{(\nu)} \rtimes {\rm{Gal}}(L/K).$$
Since ${\rm{pr}} \circ \alpha' = {\rm{id}}_{{\rm{Gal}}(L/K)}$, we have ${\rm{pr}} \circ \psi_\mathfrak{P} =  {\rm{res}}^{K^{\rm{sep}}/K}_{L/K} \circ {\rm{res}}_{K^{\rm{sep}}/K}^{(K_\mathfrak{P})^{\rm{sep}}/K_\mathfrak{P}}$. 
Mo\-re\-over,  $\mathcal{F}_p(n) / \mathcal{F}_p(n)^{(\nu)}$ is a $p_0$-group. Hence, we may apply \cite[Theorem B]{JR19} to get that ${\rm{pr}}$ has a solution 
$${\rm{Gal}}(F/K) \rightarrow \mathcal{F}_p(n) / \mathcal{F}_p(n)^{(\nu)} \rtimes {\rm{Gal}}(L/K)$$ 
such that, for every prime $\mathfrak{P} \in \mathcal{S}$, the completion of $F$ at $\mathfrak{P}$ is the fixed field in $(K_\mathfrak{P})^{\rm{sep}}$ of ${\rm{ker}}(\psi_\mathfrak{P})$. As the latter is the completion of $L$ at $\mathfrak{P}$ (for every prime $\mathfrak{P}$ of $K$), this shows that \eqref{dagger} also holds in the case $p=p_0$, thus ending the proof of Theorem \ref{thm:intro_1}.

\bibliography{Biblio2}
\bibliographystyle{alpha}

\end{document}